\documentclass[ a4paper, 12pt]{article}
\usepackage[text={16cm,24cm},centering]{geometry} 
\usepackage[utf8x]{inputenc}
\usepackage[english]{babel}
\usepackage{amssymb ,amsmath, amsthm}
\usepackage{graphicx} 
\usepackage{ccaption} 
\usepackage{graphicx}
\usepackage{hyperref}
\usepackage{multirow}
\usepackage{multirow}
\usepackage{array}
 \usepackage{amsmath}
\usepackage{multicol}
\newtheorem{theorem}{Theorem}[section]
\newtheorem{corollary}[theorem]{Corollary}

\newtheorem{lemma}[theorem]{Lemma}
\newtheorem{proposition}[theorem]{Proposition}

{\theoremstyle{remark}\newtheorem{remark}[theorem]{Remark}}
{\theoremstyle{example}\newtheorem{example}[theorem]{Example}}

\begin{document}

\renewcommand{\theequation}{\arabic{equation}}
\renewcommand{\thefigure}{\arabic{figure}}
\renewcommand{\thesection}{\normalsize\arabic{section}}
\newcommand{\ssection}[1]{%
  \section[#1]{\small\scshape #1}}


\title{\normalsize{\textbf{Number of Polynomials Vanishing on a Basis of $S_m(\Gamma_0(N))$}}}
\author{\normalsize{\textit{Iva Kodrnja}}\\ \normalsize{\textit{Faculty of Geodesy}} \and \normalsize{\textit{Helena Koncul}} \\ \normalsize{\textit{Faculty of Civil Engineering}}}
\date{}

\flushbottom
\addtolength{\voffset}{-0.5cm}
%
\captiondelim{. }
\captionnamefont{\normalfont\footnotesize\bfseries}
\captiontitlefont{\normalfont\footnotesize}


%

%

%
%

\pagenumbering{arabic} \setcounter{page}{1}

\maketitle




\vspace{0.5cm}

\begin{footnotesize}
\textit{Abstract}. In this paper we find the number of homogeneous polynomials of degree $d$ such that they vanish on cuspidal modular forms of even weight $m\geq 2$ that form a basis for $S_m(\Gamma_0(N))$. We use these cuspidal forms to embedd $X_0(N)$ to projective space and we find the Hilbert polynomial of the graded ideal of the projective curve that is the image of this embedding.
\end{footnotesize}

\begin{footnotesize}\textit{modular forms,  modular curves, projective curves, Hilbert polynomial}\end{footnotesize}

\begin{footnotesize}MSC2020: 11F11, 05E40, 13F20\end{footnotesize}

\maketitle

\section{Introduction}
\label{intro}

Let $N>1$, $m\geq 2$ an even number and let $f_0,...,f_{t-1}$ be elements of the basis of the space of cuspidal modular forms $S_m(\Gamma_0(N))$ of weight $m$ with $\dim S_m(\Gamma_0(N))=t$. Let $X_0(N)$ be the modular curve for $\Gamma_0(N)$. As in \cite{Muic1} we look at the holomorphic map $X_0(N)\to \mathbb{P}^{t-1}$ defined by
\begin{equation}\label{map}
\mathfrak{a}_z\mapsto (f_0(z):\cdots:f_{t-1}(z))
\end{equation} 

and we denote the image curve of this map by $$\mathcal{C}(N,m)\subseteq \mathbb{P}^{t-1}.$$ 

Let us set $g=f_{t-1}$. Then the map $(\ref{map})$ can be written as 
\begin{equation}\label{map1}
\mathfrak{a}_z\mapsto (f_0(z)/g(z):\cdots :f_{t-1}(z)/g(z))
\end{equation} 
and the map is a rational map of algebraic curves. 

We gather the following facts about this image curve $\mathcal{C}(N,m)$:

\begin{lemma}\label{l1}
Assume that $m \geq 2 $ even. Let $t=\dim S_m(\Gamma_0(N))$, $f_0,\dots,f_{t-1}$ be
a basis of $S_m(\Gamma_0(N))$, $g$ be the genus of $\Gamma_0(N)$ and we denote by
 $$\mathcal{C}(N,m) = C(f_0,\cdots, f_{t-1})$$
the image of the map $\ref{map}$. Then we have the following
\begin{itemize}
	\item[i)]   $\mathcal{C}(N,m)$ is an irreducible smooth projective curve in $\mathbb{P}^{t-1}$.
	\item[ii)] If $m\geq 4$, then we always have $t \geq g + 2$ and the degree of the curve is $t + g - 1$.
	\item[iii)] If  $m=2$ then $t=g$ and if $X_0(N)$ is not hyperelliptic then the map $(\ref{map})$ is the canonical embedding and the degree of the curve is $2g-2$.
	\item[iv)] If  $m=2$ then $t=g$ and if $X_0(N)$ is hyperelliptic, then the canonical map $(\ref{map})$  is of degree 2 and the degree of the curve is $g-1$.
\end{itemize}.

\end{lemma}

\begin{proof}
The i) part of Lemma \ref{l1} follows from Chow's theorem, ii) and its proof can be found in \cite{Muic1}, Corollary 3-4a and iii) and iv) for canonical map see \cite{Miranda}, Chapter 7.
\end{proof}

 Planar curves, for instance planar models of modular curves $X_0(N)$ are defined by one irreducible equation, defining equation for the curve, (see \cite{Anni}, \cite{Muic1}, \cite{Muic2}, \cite{MuKo}, \cite{Gal}). The degree of this minimal polynomial  is the degree of the curve. But when we look at maps to higher dimensional projective spaces as in our case ($\ref{map}$), then the image curve is a projective curve, a one-dimensional variety in $\mathbb{P}^{t-1}$, for $t\geq 3$. In this higher-dimensional space curves are no longer given by one equation, since hypersurfaces are of codimension $1$, yet algebraic sets whose ideal is finitely-generated.

We can compare our situation to results from \cite{Gal} where the canonical map, given in terms of cusp forms of weight $2$ for $\Gamma_0(N)$, is used to obtain models for modular curves. Then the map is to projective space $\mathbb{P}^g$, where $g$ is the genus of the group $\Gamma_0(N)$, $g=\dim S_2(\Gamma_0(N))$, and the model for modular curve is the projective curve defined as intersection of finite number of hypersurfaces, given by system polynomial equations. The study of such projective curves are in domain of intersection theory (\cite{hartshorne}, \cite{Sha}) and ideals of commutative graded rings (\cite{cm_rings}, \cite{cocoa2}).

Let $\mathcal{P}=\mathbb{Q}\left[X_0,\dots,X_ {t-1}\right] $ be the ring of polynomials in $t$ variables and $\mathcal{P}_d=\mathbb{Q}\left[X_0,\dots,X_ {t-1}\right]_d$ subring of homogeneous polynomials of degree $d$. We regard $\mathcal{P}$ as the graded ring $\mathcal{P}=\bigoplus_{d\geq 0} \mathcal{P}_d$. 

Let $I(\mathcal{C}(N,m))\subseteq \mathcal{P}$ be the ideal of the curve $\mathcal{C}(N,m)$, radical homogenous ideal consisting of all homogenous polynomials that vanish on $\mathcal{C}(N,m)$. Then  $f\in I(\mathcal{C}(N,m))$ defines a hypersurface $\mathcal{C}(N,m)\subset V(f)$ in $\mathbb{P}^{t-1}$. 
We also have a graded structure on the ideal $I(\mathcal{C}(N,m))$, if we set
\begin{equation}
I(\mathcal{C}(N,m))_d= \mathcal{P}_d \cap I(\mathcal{C}(N,m))
\end{equation}
we get the vector space $I(\mathcal{C}(N,m))_d$ of all homogenous polynomials of degree $d$ which vanish on $\mathcal{C}(N,m)$.

\begin{lemma}
	$$I(\mathcal{C}(N,m))=\bigoplus_{d\geq 0} I(\mathcal{C}(N,m))_d$$	
\end{lemma}
\begin{proof} 
	
		A homogeneous polynomial of degree $d$ only has monomials of degree $d$. Product of two homogeneous polynomials of degrees $d_1$ and $d_2$ is again a homogeneous polynomial of degree $d_1+d_2$. 
		We can see this graded structure as vector spaces or modules, because we have
		$$\mathcal{P}_jI(\mathcal{C}(N,m))_d\subseteq I(\mathcal{C}(N,m))_{jd}$$
\end{proof}

Our goal is to compute all linearly independent homogeneous polynomials of certain given degree that vanish on the curve $\mathcal{C}(N,m)$ and the dimension of the vector space $I(\mathcal{C}(N,m))_d$. In Section \ref{sec_poly} we present the algorithm to compute homogeneous polynomials vanishing on $\mathcal{C}(N,m)$, in Section \ref{results} we present the results of computation, and in Section \ref{sec_formula} we find the formula for the number of polynomials, which we present here in the following two theorems. 

\begin{theorem}\label{T1}
	Let $t=\dim S_m(\Gamma_0(N))\geq 3$, $g$ be genus of $\Gamma_0(N)$.
	Let $N\geq 2$ be such that $g=0$ or $g=t$ and $X_0(N)$ is hyperelliptic. 
Then the number $N(t,d)$ of homogeneous polynomials of degree $d\geq 0$  is equal to
\begin{equation}\label{formula1}
N(t,d)=\binom{t+d-1}{d}-(t-1)d-1.
\end{equation}
\end{theorem}

\begin{corollary}\label{T2}
	Let $t=\dim S_m(\Gamma_0(N))\geq 3$, $g$ be genus of $\Gamma_0(N)$.
Then the number $N(t,d)$ of homogeneous polynomials of degree $d\geq 0$  is equal to
 \begin{itemize}
	\item $g=\{1,2\}$ $$N(t,d)=\binom{t+d-1}{d}-(t-1)d-1- (g+(d-2)g)$$
	\item $g>2$ $$N(t,d)=\binom{t+d-1}{d}-(t-1)d-1- ((g-2)+(d-2)(g-1))$$
\end{itemize}

\end{corollary}

We prove Theorems \ref{T1} and \ref{T2} in Section \ref{HF} by computing the Hilbert polynomial of the curve $\mathcal{C}(N,m)$. Hilbert polynomial is a great tool to  record and handle information about dimensions of homogeneous components of graded modules, (\cite{cocoa2}).

\section{Computing homogeneous polynomials vanishing on cusp forms}
\label{sec_poly}

Let $f_0,\dots, f_{t-1} \in S_m(\Gamma_0(N))$ be a basis of the space of cuspidal modular forms for the congruence subgroup $\Gamma_0(N)$ of weight $m\geq 2$. 

Let $P\in\mathbb{Q}\left[x_0,\dots,x_{t-1} \right] $ be a homogeneous polynomial of degree $d$ 

$$P(x_0,\dots, x_{t-1})=\sum_{\substack{0\leq i_0,\dots,i_{t-1}\leq d\\i_0+\dots +i_{t-1}=d}}a_{i_0,\dots,i_{t-1}}x_0^{i_0}\cdots x_{t-1}^{i_{t-1}}.$$

For a given degree $d\geq 0$, we are interested in those polynomials which vanish on the elements of the basis $f_0,\dots, f_{t-1}$,

\begin{equation}
P(f_0(z),\cdots, f_{t-1}(z))=\sum_{\substack{0\leq i_0,\dots,i_{t-1}\leq d\\i_0+\dots +i_{t-1}=d}}a_{i_0,\dots,i_{t-1}}f_0^{i_0}\cdots f_{t-1}^{i_{t-1}}=0
\end{equation}

for all $\mathfrak{a}_z\in X_0(N)$.

Vector space $\mathcal{P}_d$ of all homogeneous polynomials of degree $d$ is generated with monomials and its dimension can be viewed as the number of coefficients $a_{i_0,\dots,i_{t-1}}$ with respect to the indexing set of the set of monomials of degree $d$ $$I=\left\lbrace (i_0,\dots, i_{t-1}): 0\leq i_0,\dots,i_{t-1}\leq d,i_0+\dots +i_{t-1}=d \right\rbrace.$$ The cardinality of $|I|$ is known as the weak composition problem in combinatorics and the solution is
\begin{equation}\label{d'}
d'=\dim \mathcal{P}_d=|I|=\binom{d+t-1}{d}.  
\end{equation}  

We will order $I$ using the lexicographical ordering,(see \cite{ideals}) so that we consider a polynomial $P$ as a finite linear array of its coefficients 
\begin{equation}\label{coef}
P \longrightarrow (a_0,\dots,a_{d'-1}) 
\end{equation}

satisfying the order of corresponding monomials, as basis representation of $P$.

We are interested in subspaces  $I(\mathcal{C}(N,m))_d\subseteq \mathcal{P}_d$  containing polynomials that vanish on the basis of $S_m(\Gamma_0(N))$ for certain choices of $d, N, m$ and their dimensions,
\begin{equation}
I(\mathcal{C}(N,m))_d=\left\lbrace P\in \mathcal{P}_d : \, P(f_0,\cdots,f_{t-1})=0 \right\rbrace 
\end{equation}

where $f_0,\dots,f_{t-1}$ is a basis of $S_m(\Gamma_0(N))$. Each modular form is in practical computations given by a finitely many coefficients of its integral Fourier expansion in the cusp $\infty$. 

The polynomial combination $P(f_0,\dots,f_{t-1})$ is again a modular form, of weight $md$, where $d$ is the degree of the homogeneous polynomial $P$, since cuspidal forms on a given group also make a graded ring, $S(\Gamma_0(N))=\oplus_m S_m(\Gamma_0(N))$. 

The condition of vanishing of the modular form $P(f_0,\dots,f_{t-1})$  is known as the Sturm bound saying that we only consider a finite number $B$ of coefficients of the $q$-expansion of the form to distinguish forms,
\begin{equation}\label{sturm} B_m=\left\lfloor \dfrac{m\left[SL_2(\mathbb{Z}):\Gamma_0(N) \right] }{12}\right\rfloor.
\end{equation}

Similar to \cite{Gal}, \cite{Muic2}, \cite{MuKo} the algorithm for computing polynomials vanishing on a basis of $S_m(\Gamma_0(N))$ is based on the following linear algebra considerations: for fixed values of $d,N,m$ we are solving a homogeneous system of equations, where the unknowns are coefficients of a polynomial $P$, $a_0,\dots,a_{d'-1}$ as in $(\ref{coef})$ and the coefficients of the system are values of $q$-expansions of evaluated monomials $f_0^{i_0}\dots f_t^{i_{t-1}}$ over the indexing set $I$,
\begin{align*}
P(f_0,\dots,f_{t-1})=&\sum_{\substack{0\leq i_0,\dots,i_{t-1}\leq d\\i_0+\dots +i_{t-1}=d}}a_{i_0,\dots,i_{t-1}}f_0^{i_0}\cdots f_{t-1}^{i_{t-1}}\\
=&\sum_{\substack{0\leq i_0,\dots,i_{t-1}\leq d\\i_0+\dots +i_{t-1}=d}}a_{i_0,\dots,i_{t-1}}\left(a_0^{(i_0,\dots,i_{t-1})} +a_1^{(i_0,\dots,i_{t-1})}q+\dots \right)\\
=&p_0+p_1q+p_2q^2+ \cdots.
\end{align*}

The homogeneous system is $p_0=p_1=\cdots=p_{B_{md}}=0$
and its solutions are obtained as the basis of the right kernel of the transpose of $d'\times B_{md}$ matrix whose rows are made of coefficients of $f_0^{i_0}\dots f_{t-1}^{i_{t-1}}$, after ordering the index set $I$.

Here is the algorithm, for a given $N$ and weight $m$, with the use of lexicographic ordering on the set of monomials of degree $d$:

\begin{itemize}
	\item[\textbf{Input:}] $q$-expansions of $f_0,\dots,f_{t-1}$ basis of $S_m(\Gamma_0(N))$
	\item for a degree $d\geq 0$: 
	\begin{itemize}
	        \item    for each monomial index $(i_0,\dots,i_{t-1})\in I$ in the ordered set of monomials of degree $d$:\
	        
	        \begin{itemize}
	        	\item[] compute $f_0^{i_0}\dots f_{t-1}^{i_{t-1}}$
	        \end{itemize}
	           
	          \item create a $d'\times B_{md}$ matrix $A$,  whose rows are first $B_{md}$ coefficients of $q$-expansion of  $f_0^{i_0}\dots f_{t-1}^{i_{t-1}}$ 
	           \item return the dimension (or the elements) of the right kernel of $A$ 
	 \end{itemize}          
	 \item [\textbf{Output:}]  beginning part of an array containing the number of linearly independent homogeneous polynomials of degree $d\geq 0$ vanishing on all forms, i.e.~such that $P(f_0,\dots,f_{t-1})=0$.
\end{itemize}

We can also compute explicitly all of these polynomials. 
In our computations we are using the SAGE software system \cite{sage} and the cusp form basis we are using is generated by command
\begin{verbatim}
CuspForms(Gamma_0(N),m).q_integral_basis(prec)
\end{verbatim}


	

\section{Results}
\label{results}
Let $t=\dim S_m(\Gamma_0(N))$, $g$ be the genus of $\Gamma_0(N)$.
Using the algorithm of Section \ref{sec_poly}, we were able to compute homogeneous polynomials that vanish on all elements of basis of $S_m(\Gamma_0(N))$, for small degrees up to $7$ or at times lower. We order results by number of elements in the basis, $2\leq t\leq 8$. We will denote hyperelliptic modular curves $X_0(N)$ with an asterix. From  \cite{Ogg} we know that  $X_0(N)$ is a  hyperelliptic curve if and only if   $$N\in \{22, 23, 26, 28, 29, 30,  31, 33, 35, 37, 39, 40, 41, 46, 47, 48, 50, 59, 71\}.$$
  

We start with the case $t=2$. We have $4$ cases with $g=0$, one for $g=1$ case and $8$ hyperelliptic cases when $g=2$:


\begin{itemize}
		\item genus 0 : $(2,12), (2,14), (3,10), (4,8)$
	\item genus 1 : $(11,4)$
	\item genus 2 : $(22,2)^*, (23,2)^*, (26,2)^*, (28,2)^*, (29,2)^*, (31,2)^*,\\ (37,2)^*, (50,2)^*$ 
\end{itemize}

\begin{proposition}\label{prop1}
Let $t=\dim S_m(\Gamma_0(N))=2$ and $\left\lbrace f_0,f_1 \right\rbrace $ be the basis of $S_m(\Gamma_0(N))$. There are no homogeneous polynomials in two variables of degree $d=0$ to $d=7$ such that $P(f_0,f_1)=0$. 	
\end{proposition}

\subsection{$t=3$}

 Since $X_0(N)$ is mapped by $(\ref{map})$ to the projective plane $\mathbb{P}^2$ its image is a planar projective curve hence given by one irreducible defining equation, since curves are hypersurfaces of projective plane. The degree for which this happens is then the degree of the curve and for all higher degrees we can find more than one polynomial vanishing on the curve. These higher degree polynomials are reducible, because they have the defining polynomial as a factor.

There are $16$ cases for $g=0$ and $11$ for $g=3$ of which $7$ are hyperelliptic: 

\begin{itemize}
	\item genus 0 : (2,16), (2,18), (3,12), (3,14), (4,10), (5,8), (5,10), (6,6), (7,6), (7,8), (8,6), (9,6), (10,4), (12,4), (13,4), (16,4)
	\item genus 3: $(30,2)^*$, $(33,2)^*$ (34,2), $(35,2)^*$, $(39,2)^*$, $(40,2)^*$,  $(41,2)^*$, (43,2), (45,2), $(48,2)^*$, (64,2)
\end{itemize}

\begin{proposition}\label{prop2}
	Let $t=\dim S_m(\Gamma_0(N))=3$ and  $\left\lbrace f_0,f_1,f_2 \right\rbrace $ be the basis of $S_m(\Gamma_0(N))$. In Table $\ref{3formeirreduc.}$ we give the number of all linearly independent homogeneous polynomials up to degree $7$ and the number of irreducible ones among them that satisfy $P(f_0,f_1,f_2)=0$.

	\begin{table}[ht]
		\caption{Number of polynomials for $\dim S_m(\Gamma_0(N))=3$}
		\begin{center}
			\begin{tabular}{|c||c|c|c|c|c|c|c|c|}
				\hline
			
				\multirow{2}{*}{g}&
				\multicolumn{8}{c|}{degree of $P$} \\
				& 0&1 & 2& 3 & 4 & 5 & 6&7 \\
				\hline
				\hline
			0&0&0& 1 & 3 & 6 & 10 & 15 & 21 \\
				\hline
		 irred.&0&0&1 & 0 & 0 & 0 & 0 & 0 \\
				\hline\hline
			3&0&0 & 0 & 0  & 1 & 3 & 6 & 10 \\
				\hline
				irred.  & 0&0&0 & 0 & 1 & 0 & 0 & 0 \\
				\hline\hline
						 3, hypell& 0&0&1 & 3 & 6 & 10 & 15 & 21 \\
			\hline
			irred. &0&0& 1 & 0 & 0 & 0 & 0 & 0 \\	
			\hline
			\end{tabular}
		\end{center}
		\label{3formeirreduc.}
	\end{table}

\begin{table}[ht]
	\caption{Equations for non-hyperelliptic $X_0(N)$ with $g=3$ from weight $2$ cusp forms}
	\begin{center}
		{\renewcommand{\arraystretch}{1.4}
			\begin{tabular}{c|l}
				
				$N=34$& $x^{2} y^{2} - 2 \, x y^{3} + y^{4} - x^{3} z - 3 \, x y^{2} z - 4 \, y^{3} z + 3 \, x^{2} z^{2} + 3 \, x y z^{2} $\\&$+ 6 \, y^{2} z^{2} - 4 \, x z^{3} - 4 \, y z^{3} + 2 \, z^{4}$\\
				\hline
				$N=43$&$x^{2} y^{2} - x y^{3} + 9 \, y^{4} - x^{3} z - 2 \, x^{2} y z - 3 \, x y^{2} z - 24 \, y^{3} z$\\ & $ + 2 \, x^{2} z^{2} + 5 \, x y z^{2} + 28 \, y^{2} z^{2} - 3 \, x z^{3} - 16 \, y z^{3} + 4 \, z^{4}$\\
				\hline
				$N=45$&$x^{2} y^{2} - x^{3} z + y^{3} z - x y z^{2} + 5 \, z^{4}$
		\end{tabular}}
	\end{center}
	\label{jdbe_3}
\end{table}

\end{proposition}
We see that in the hyperelliptic case the numbers coincide with those for genus $0$ cases. The numbers appearing in Proposition $\ref{prop2}$ are the initial part of the well-known integer sequence called triangle numbers, \cite{oeis}. They also appear in the usual genus-degree formula for curves (\cite{Anni}, Theorem 2.1).
This happens because to raise the degree we multiply a polynomial with the monomial $x_i$, $i=\{0,1,2\}$.

In \cite{Muic2} and \cite{MuKo} we have a formula relating degree $d$ of the image curve $\mathcal{C}(N,m)$ and the degree  $d(f_0,f_1,f_2)$ of the map $(\ref{map})$
\begin{equation}\label{formula_deg}
d\cdot d(f_0,f_1,f_2)=\dim S_m(\Gamma_0(N))+ g(\Gamma_0(N))-1 -\epsilon_{m},  
\end{equation}	
where $\epsilon_2=1$  and $\epsilon_m=0$ for $m\ge 4$ is the number of possible common zeroes of the basis cusp forms. 
Given $t=3$, the right-hand side of $(\ref{formula_deg})$ can attain values $3+0-1-0=2$ for $g=0$, because we must have $m>2$, or $3+3-1-1=4$ for $g=3$ and $m=2$.
Since we have computed irreducible equations for $\mathcal{C}(N,m)$ of that exact degree we can conclude that the map is birational.

\begin{corollary}
Assume that $\dim S_m(\Gamma_0(N))=3$ and let $\left\lbrace f_0,f_1,f_2 \right\rbrace $ be the basis of $S_m(\Gamma_0(N))$. Then the map $X_0(N)\to\mathbb{P}^2$ given by $$\mathfrak{a}_z\mapsto (f_0(z):f_1(z):f_2(z))$$
is birational equvalence of $X_0(N)$ and the image curve $C(f_0,f_1,f_2)$. The curve $C(f_0,f_1,f_2)$ is a quadric if $g(X_0(N))=0$ or $g(X_0(N))=3$ and the curve is hyperelliptic or a quartic if $g(X_0(N))=3$. The quartic equations (in homogeneous coordinates $(x:y:z)$) are given in Table $\ref{jdbe_3}$.
\end{corollary}


\subsection{$4\leq t\leq 8$}

 \begin{table}[ht]
	\caption{ $(N,m)$ for $4\leq\dim S_m(\Gamma_0(N))\leq 8$}
	\begin{center}
		\begin{tabular}{|c|c|l|}
		\hline
			t&g&(N,m)\\
			\hline\hline
			4&0&(2,20), (2,22), (3,16) (4,12)\\
			\cline{2-3}
			&1&(14,4), (15,4), (17,4), (19,4), (11,6)\\
			\cline{2-3}
			&4&(38,2), (44,2), $(47,2)^*$, (53,2), (54,2), (61,2), (81,2)\\

			\hline\hline
			5&0&(2,24), (2,26), (3,18), (3,20), (4,14), (5,12), (5,14), (6,8), (7,10),\\
			
			&& (8,8), (9,8), (10,6), (13,6), (18,4), (25,4)\\
			\cline{2-3}
			&2&$(23,4)*$\\
			\cline{2-3}
			&5&(42,2),$(46,2)^*$, (51,2), (52,2), (55,2), (56,2), (57,2) $(59,2)^*$,\\
			&& (63,2), (65,2), (67,2), (72,2), (73,2), (75,2)\\
			\hline
			\hline
			6&0&(2,28),(2,30), (3,22), (4,16)\\
			\cline{2-3}
			&1&(11,8), (17,6), (20,4), (21,4), (27,4)\\
			\cline{2-3}
			&6&(58,2), $(71,2)^*$ (79,2) \\
			\hline\hline
			7&0&(2,32), (2,34), (3,24), (3,26), (4,18), (5,16), (5,18) (6,10), (7,12), \\
			&& (7,14), (8,10), (9,10), (12,6), (13,8), (16,6)\\
			\cline{2-3}
			&2&$(22,4)^*$, $(29,4)^*$, $(31,4)^*$\\
			\cline{2-3}
			&7&(60,2), (62,2), (68,2), (69,2), (77,2), (80,2), (83,2), (85,2), (89,2),\\
			&& (91,2), (97,2), (98,2)\\
			
			\hline\hline
			8&0&(2,36), (2,38), (3,28), (4,20)\\
			\cline{2-3}
			&1&(11,10), (14,6), (15,6), (19,6), (24,4), (32,4)\\
			\cline{2-3}
			&8&(74,2), (76,2) \\
			\hline
					
		\end{tabular}
	\end{center}
	\label{t=45678}
\end{table}

For $t=4$ there are $4$ cases for $g=0$, $5$ for $g=1$ and $7$ for $g=4$ cases of which one is hyperelliptic;





for $t=5$ there are $15$ cases of $g=0$, one hyperelliptic case of $g=2$ and $14$ for $g=5$ of which two are hyperelliptic;
			
%




for $t=6$ we have $4$ cases of $g=0$, $5$ for $g=1$ and $3$ cases for $g=6$ of which one is hyperelliptic;







for $t=7$ we have $15$ cases for $g= 0$, $3$ hypereliptic cases for $g=2$ and $12$ for $g=7$ cases;

				






and finally for $t=8$ we have $4$ cases for $g=0$, $6$ for $g=1$ and $2$ cases for $g=8$, all of which are in Table \ref{t=45678}.





\begin{proposition}\label{prop3}
	Let $4\leq t\leq 8$. In Table $\ref{tablica_brojevi}$ we see the numbers of all linearly independent homogeneous polynomials of degree $2\leq d\leq 8$ that vanish on the basis of $\dim S_m(\Gamma_0(N))$, $P(f_0,f_1,\cdots,f_{t-1})=0$. There are no such polynomials for $d< 2$.

\end{proposition}

\begin{table}[ht]
	\caption{Number of polynomials for $4\leq t\leq 7$ and $2\leq d\leq 7$}
	\begin{center}
		\begin{tabular}{|r||c||c|c|c|c|c|c|}
			\hline
			\multirow{2}{*}{t}&
			
			\multirow{2}{*}{g}&
			\multicolumn{6}{c|}{degree d of $P$} \\
			&	& 2& 3 & 4 & 5 & 6&7 \\
			\hline
			\hline
						
			\multirow{3}{*}{4}&	0& 3 & 10 & 22 & 40 & 65 & 98 \\
			
			&1 & 2 & 8  & 19 & 36 & 60 & 92 \\
			
			&4& 1 & 5 &14 & 29 & 51 & 81 \\
			\hline
			\multirow{3}{*}{5} &	0& 6 & 22 & 53 & 105 & 185 & 301 \\
			
			&	2 & 4 & 18  & 47 & 97 & 175 & 289 \\
			
			&	5& 3 & 15 &42 & 90 & 166 & 278 \\
			\hline
			\multirow{3}{*}{6}&	0& 10 & 40 & 105 & 226 & 431 & 756 \\
			
			&	1 & 9 & 38  & 102 & 222 & 426 & 750 \\
			
			&	6& 6 &31 &91 & 207 & 407 & 727 \\
			\hline	
			\multirow{3}{*}{7}
			
			& 0& 15 & 65 & 185 & 431 & 887 & 1673 \\
			
			&		2 & 13 & 61  & 179 & 423 & 877 & 1661 \\
			
			&	7& 10 &54 &168 & 408 & 858 & 1638 \\
			
			\hline
			\multirow{3}{*}{8}
			&	0& 21 & 98 & 301 & 756 & 1673 & 3382 \\
			
			&	1 & 20 & 96  & 298 & 752 & 1668 & 3376 \\
			
			&	8& 15 &85 &9281 & 729 & 1639 & 3341 \\
			
			\hline
		\end{tabular}
	\end{center}
	\label{tablica_brojevi}
\end{table}

\section{General formula and a special triangle of numbers}
\label{sec_formula}
In this section we will find the recursion and the general formula that govern the numbers from Table \ref{tablica_brojevi}. We will do this for the case $g=0$ and then for higher $g$ we will show the relation to $g=0$ case. The first $7$ numbers for $3\leq t\leq 8$ in this series for $g=0$ cases are given in Table \ref{0genusi}. We omit the first row and column of zeros for $t=2$ and $d=1$.

\begin{table}[ht]
	\caption{Number of polynomials for $g=0$}
	\begin{center}
		\begin{tabular}{|c||c|c|c|c|c|c|c|}
			\hline
			
			\multirow{2}{*}{$t$}&
			\multicolumn{6}{c|}{$d$} \\
			& 2& 3 & 4 & 5 & 6&7 \\
			\hline
			\hline
			3 & 1 &3 &6 &10  &15  &21 \\
			\hline
			4 & 3 & 10 & 22 &40  &65  &98 \\
			\hline
			5 & 6 & 22 & 53 & 105 & 185 &301  \\
			\hline  
			6 & 10 & 40 & 105 & 226 & 431& 756 \\
			
			\hline
			7 & 15 & 65 & 185 & 431 &887 &1673  \\
			\hline
			8 & 21 & 98 & 301 & 756 &1673 & 3382 \\
			\hline
			
		\end{tabular}
	\end{center}
	\label{0genusi}
\end{table}

In Table \ref{0genusi} we see notice the following rule: \textit{a particular number $N(t,d)$ can be calculated from previous elements if we add the number in previous row, number in previous column and add $(t-2)(d-1)$ }. 

This can be illustrated as an extended triangle shown in Figure \ref{trokut}, where we add between numbers from Table \ref{0genusi} as antidiagonals and diagonals all multiples of all numbers. For instance, in figure \ref{trokut} we see $105=40+12+53$.

This adding rule can be written as a non-homogeneous recursion
\begin{equation}\label{rekurzija}
N(t,d)=N(t-1,d)+N(t,d-1) +(t-2)(d-1).
\end{equation}

\begin{figure}\label{trokut}
\begin{tabular}{ccccccccccccccccccccc}
	&& &    &    &    &    & &0&\\
&& &    &    &    &    &  0&\textbf{1}&0\\	
&& &    &  &  &  0  &  \textbf{2}  &  1&\textbf{2}&0\\
&& &   & &   0 &  \textbf{3}  &  3 &  \textbf{4}  &  3&\textbf{3}&0\\
&& & & 0  &  \textbf{4}  & 6 &   \textbf{6} &  10 & \textbf{ 6 } &  6&\textbf{4}&0\\
& &&0& \textbf{5}   &  10 &  \textbf{8}  &  22 & \textbf{9}   &  22 &  \textbf{8}  &  10&\textbf{5}&0\\
&&0 &\textbf{6}&  15 &  \textbf{10}  &  40 &  \textbf{12 } &  53 &  \textbf{12}  &  40 &  \textbf{10}  &  15&\textbf{6}&0\\
&0&\textbf{7}& 21& \textbf{12}  &  65  & \textbf{15} & 105   & \textbf{16}  & 105   & \textbf{15}  & 65   & \textbf{15} &  21&\textbf{7}&0\\
\end{tabular}
\caption{Extended triangle from Table \ref{0genusi}}
\end{figure}

For the homogeneous part of recursion $(\ref{rekurzija})$, that is also defining the Pascal triangle, $$N_h(t,d)=N_h(t-1,d)+N_h(t,d-1) $$ we can find the generating function \cite{Wilf}
$$
N_t(x)=\sum_{k\geq 0} N_h(t,d)x^k =N_{t-1}(x)+xN_t(x) \Rightarrow$$
$$N_t(x)=\dfrac{1}{1-x}N_{t-1}(x)=\dfrac{1}{(1-x)^t}$$
	
where $N_h(t,d)$ is the coefficient of $x^t$ in $N_t(x)$ is $$N_h(t,d)=\binom{t+d-1}{d},$$

precisely the size $|I|$ of set of homogeneous monomials of degree $d$, $(\ref{d'})$.

We continue by finding a particular solution in the form:
\begin{align*}
N(t,d)_p&=(C_1t+C_2)(D_1d+D_2)+E\\
&=C_1D_1td+C_1D_2t+D_1C_2d+C_2D_2+E
\end{align*}

and after substitution in $(\ref{rekurzija})$ we get
$$C_1=-1,\,C_2=1,\,D_1=1,\,D_2=0,\, E=-1. $$

\begin{lemma}\label{formula_rek}
Numbers in Table \ref{0genusi} satisfy the formula:
\begin{equation} \label{fr}
N(t,d)=\binom{t+d-1}{d} -(t-1)d-1
\end{equation}

\end{lemma}

\begin{remark}
Antidiagonals of our Table \ref{0genusi}, read as series $1,\, 3, 3,\, 6, 10, 6, ....$ can be found in \textit{The on-line encyclopedia of integer sequences} OEIS, \cite{oeis}, database as sequence A124326 with the description  \begin{equation}\label{PRas}
T(n,m)=A007318(n,m)-A077028(n,m)
\end{equation}
 as a difference between the Pascal triangle, listed in OEIS as A007318, with defining formula $C(n,m)=\binom{n}{m},\ 0\leq m\leq n$ and the Rascal triangle, \cite{rascal}, A007028 in OEIS, $T(n,m)=m(n-m)+1,\ n\geq 0,\ 0\leq m\leq n$ 
\begin{equation}\label{Pas}
T(n,m)=\binom{n}{m}-m(n-m)-1
\end{equation}
If we substitute $n=t+d-1$ and $m=d$ in $(\ref{Pas})$ we get formula $(\ref{fr})$ from Lemma \ref{formula_rek}.

Alternatively, we can also say that rows of our Table \ref{0genusi} are the diagonals of the triangle A124326, see Figure 1.

\end{remark}

\begin{remark}
The numbers in Table \ref{0genusi} are generated via modular forms but according to the rule in Section \ref{sec_formula} they can be constructed in another way from Figure \ref{trokut} by relating it to sequence A003991 from OEIS database, which is the table of multiplication read by antidiagonals defined by $T(n,m)=nm,\ n,m \geq 1$ and the Pascal rule.

For $n,m\geq 0$ and setting the initial values $T(n,0)=T(0,m)=0$, the not highlighted numbers of triangle in Figure \ref{trokut} satisfy the recursion and formula
\begin{align*}
T(n,m)=T(n-1,m)+T(n,m-1)+nm\\
T(n,m)=\binom{n+m+2}{m+1} -(n+1)(m+1)-1.
\end{align*}

Again by substituting $\overline{n}=n+m+2$
 and $\overline{m}=m+1$ we obtain \ref{PRas} which is the description of the sequence A124326 in OEIS database.

This way of obtaining the sequence A124326 appears to be new.	

\end{remark}


\section{Hilbert function and Hilbert polynomial}
\label{HF}

In Section \ref{intro} we have seen that homogeneous polynomials of degree $d$ vanishing on the curve $\mathcal{C}(N,m)$ make the vector space 
\begin{equation*}
I(\mathcal{C}(N,m))_d= \mathcal{P}_d \cap I(\mathcal{C}(N,m))
\end{equation*}

and generate the graded ideal
	$$I(\mathcal{C}(N,m))=\bigoplus_{d\geq 0} I(\mathcal{C}(N,m))_d.$$

From Lemma \ref{formula_rek} we have: 

\begin{corollary}
	Let $2\leq d\leq 7$. Then $I(\mathcal{C}(N,m))_d$ is a vector space of dimension
\begin{equation}\label{dimenzija}
\dim I(\mathcal{C}(N,m))_d=\binom{t+d-1}{d}-(t-1)d-1.	
\end{equation}
\end{corollary}

For $d\geq 0$ we define the Hilbert function \cite{cocoa2} of the curve $\mathcal{C}_{N,m}$ to be the Hilbert function of its coordinate ring:
\begin{equation}\label{Hil_fun}
HF_{\mathcal{C}_{N,m}}(d)=HF_{\mathcal{P}/I(\mathcal{C}_{N,m})}(d)=\dim \mathcal{P}_d - \dim I_d
\end{equation}

For the polynomial ring $\mathcal{P}$ we have 
\begin{equation}\label{Hil_fun_P}
HF_{\mathcal{P}}(d)=\dim \mathcal{P}_d=\binom{t+d-1}{d}
\end{equation}

By Hilbert-Serre theorem (see \cite{hartshorne}, Thm. 7-5) for a projective curve there is a unique polynomial of degree one such that for $d\gg 0$ $$HP_{\mathcal{C}_{N,m}}(d)=HF_{\mathcal{C}_{N,m}}(d).$$

We have calculated $\dim I_d$ up to $d=7$, but we now use formulas from Section \ref{sec_formula} to find this polynomial explicitly.

\begin{proof}[ Proof of Theorem \ref{T1}]

In the case of genus $0$ or a hyperelliptic $X_0(N)$ from $(\ref{Hil_fun})$ , $(\ref{Hil_fun_P})$ and $\ref{dimenzija}$ we have:
\begin{align}
HF_{I(C)}(d)&=\binom{t+d-1}{d} - \binom{t+d-1}{d}+(t-1)d+1 \nonumber\\
&=(t-1)d+1 \label{Hil_fun1}
\end{align}	

From \ref{Hil_fun1} we can see that the Hilbert function is expressed as the Hilbert polynomial from the index $d\geq 2$ and the degree of the polynomial is equal to one, as is always in the case of curves.
 
From this, we can conclude that this will be the Hilbert function for all $d$ since, hence Theorem \ref{T1} follows from here. 
	
\end{proof}

The coefficients and the degree of this linear Hilbert polynomial for our $\mathcal{C}(N,m)$ as an embedded projective variety have following interpretation, \cite{cocoa2}:
\begin{itemize}
\item $\dim \mathcal{C}(N,m)=1$ is the degree of the Hilbert polynomial 
\item the degree of $\mathcal{C}(N,m)$ is the leading coefficient of the Hilbert polynomial.
\end{itemize}

\begin{example}
For $t=3$ we have irreducible equations for projective curves in $\mathbb{P}^2$, as we see in Proposition \ref{prop2} and the Hilbert polynomial is $$HP_{\mathcal{C}(N,m)}(d)=2d+1$$ for genus $0$ curves and hyperelliptic genus $3$ curves so we see that the degree is $2$ and for nonhyperelliptic curves of genus $3$ we have $$HP_{\mathcal{C}(N,m)}(d)=2d+1+1+(d-2)2=4d-2$$ and we see that the degree of the curve is $4$. 
\end{example}

\begin{corollary}
The degree of the curve $\mathcal{C}(N,m)$ is:
\begin{equation*}\label{deg}
\deg\mathcal{C}(N,m)=\begin{cases}
t-1& :  g(X_0(N))=0 \text{ or } g(X_0(N))=t\\
& \text{ and } X_0(N) \text{ is hyperelliptic},\\
t+g-1& : g(X_0(N))\in\left\lbrace 1,2  \right\rbrace, \\
t+g-2& : g(X_0(N))\ge 2.
\end{cases} 
\end{equation*} 	
	
\end{corollary}

We see that this agrees with Lemma \ref{l1} from Section \ref{intro}.

Acknowledgements. This work is supported (in part) by the Croatian Science Foundation under the project number HRZZ-IP-2022-10-4615.

\end{document}